\documentclass[xcolor,table,nobib]{tufte-handout}

\usepackage{matyi}
\theoremstyle{remark}
\newtheorem*{remarks}{Remark}
\usepackage{float}
\usepackage{tkz-euclide}
\usepackage[backend=biber,style=alphabetic]{biblatex}
\date{}

\author{Sovanlal Mondal\\
\vspace{0.5 cm}
  \normalsize Department of Mathematical Sciences,
  University of Memphis\\
  Memphis, TN 38152, USA \\
  email \href{mailto:smondal@memphis.edu}{smondal@memphis.edu}} \title{Grid method for divergence of
  averages}
\begin{document}
\maketitle
\begin{abstract} 
In this paper, we will introduce the `grid method' to prove that the
  extreme case of oscillation occurs for the averages
  obtained by sampling a flow along the sequence of times of the form
  $\{n^\alpha: n\in \setN\}$, where $\alpha$ is a \emph{positive
    non-integer rational number}. Such behavior of a sequence is known as the `strong sweeping out' property. By using the same method, we will give an example of a general class of sequences which satisfy the  \emph{strong sweeping out} property. This class of sequences may be useful to solve the longstanding open problem: for a given \emph{irrational} $\alpha$, whether the sequence $(n^\alpha)$ is \emph{bad} for pointwise ergodic theorem in $L^2$ or not. In the process of proving this
  result, we will also prove a continuous version of the Conze principle.
\end{abstract}



\tableofcontents
\section{Introduction and main results}
Let $(T^t)_{t\in \setR}$ be a measure preserving flow on a Lebesgue probability space
$(X,\Sigma, \mu)$. Birkhoff's pointwise ergodic theorem asserts that for any function $f\in L^1(X)$, the Cesàro averages $\displaystyle\frac{1}{N}\sum_{n\in [N]}f(T^{n}x)$ converge for almost every $x$, where $[N]=\{1,2,\dots ,N\}$. This classical result motivated others to study the ergodic averages along a sequence of positive real numbers $(s_n)$ i.e. $\displaystyle\frac{1}{N}\sum_{n\in [N]}f(T^{s_n}x)$.
\begin{defn}[label={defn:1}]{}{} Let $1\leq
p\leq \infty$. A sequence $(s_n)$ of positive real numbers is said to be
 \emph{pointwise good} for $L^p$ if for 
every system
$(X,\Sigma, \mu, T^t)$ and every $f\in L^p(X)$,
$\displaystyle\lim_{N\to \infty}\frac{1}{N}\sum_{n\in [N]}f(T^{s_n}x)$ exists for almost every $x\in X.$\\
\end{defn} 
\begin{defn}{}{}
A sequence $(s_n)$ of positive real numbers is said to be
\emph{pointwise bad} in $L^p$ if for every aperiodic system
$(X,\Sigma, \mu, T^t)$, there is an element $f\in L^p(X)$ such that
$\displaystyle\lim_{N\to \infty} \dfrac1{N}\displaystyle\sum_{n\in
[N]}f(T^{s_n}x)$ does not exist for almost every $x\in X.$
\end{defn}
The behavior of ergodic averages along a subsequence has a rich history. First breakthrough result in this direction was due to Krengel \autocite{Krengel} who showed that there exists a sequence of positive integers which is {pointwise bad} for $L^\infty$. A few years later, Bellow \autocite{Bellow1} proved that any lacunary sequence, for example $(s_n)=(2^n)$, is {pointwise bad} for $L^p$ when $p\in [1,\infty)$. On the other extreme, if a sequence grows slower than any positive power of $n$, for example $(s_n)=\big((\log n)^c\big)$, $c>0$, then it is also {pointwise bad} for $L^\infty$ \autocite[Theorem 2.16]{JW}. In \autocite{Bellow2} and \autocite{Karin}, it was proved that whether a sequence will be  {pointwise good} for $L^p$ or not also depends on the value of $p$. More precisely, they showed that for any given $1\leq p<q\leq \infty$, there are sequences $(s_n)$ which are {pointwise good} for $L^q$ but {pointwise bad} for $L^p$ (see also \cite{Andrew} for finer results in terms of Orlicz spaces). There are many instances where the behavior of the averages cannot be determined by either the growth rate of the sequence $(s_n)$ or the value of $p$; instead one has to analyze the intrinsic arithmetic properties of the sequence $(s_n)$. One such curious example is $(n^\alpha)$.
A celebrated result of Bourgain \autocite[Theorem 2]{BO} says that the sequence $(n^\alpha)$ is {pointwise good} for $L^2$ when $\alpha$ is a {positive integer}, on the other hand $\lfloor (n^\alpha+\log n)\rfloor$ is known to be {pointwise bad} for $L^2$ when $\alpha$ is a \emph{positive integer} \autocite[Theorem C]{BKQW}. Later, in a series of papers, it was established that for any polynomial $P(x)$, the sequence $P(n)$, and the sequence of primes are {pointwise good} for $L^p$, when $p>1$ \autocite{BO3},\autocite{Wierdl}. However, they are {pointwise bad} for $L^1$ \autocite{LaVictoire2}, \autocite{BuczolichMauldin}. Thus, the $L^1$ case turned out to be more subtle than the others. It was largely believed that there cannot be any sequence $(s_n)$ which is  {pointwise good} for $L^1$ and satisfies $(s_{n+1}-s_n) \to \infty$ as $n \to \infty$. Buczolich \autocite{Buczolich} constructed inductively a sophisticated example to disprove this conjecture. Later, it was shown in \autocite{UrbanZien} that  $\lfloor n^c\rfloor, c\in(1,1.001)$ is {pointwise good} for $L^1$. The current best result is due to Mirek who showed that $\lfloor n^c\rfloor, c\in(1,\frac{30}{29})$ is {pointwise good} for $L^1$ \autocite{Mirek}, see also \autocite{Trojan}. It would be interesting to know if the above result can be extended to all \emph{positive non integer }$c$. For further exposition in this area, the reader is referred to the survey article of \autocite{RW}.\\

 \ \ \ \ In this paper we will be concerned with the behavior of ergodic averages along $(n^\alpha)$ where $\alpha$ is a \emph{positive non-integer rational number}. In a strong contrast with \autocite[Theorem 2]{BO}, Bergelson, Boshernitzan and Bourgain proved in \autocite{BBB} that
for any \emph{ non integer rational $\alpha$}, $(n^\alpha)$ is { pointwise bad } for $L^\infty$. Their proof is based on Bourgain's entropy method. Later, the result was
improved and the proof was simplified in \autocite[Example 2.8]{JW}: it was proved
that if the averages are taken along $(n^\frac{a}{b})$, where
$a,b\in \setN$ and $b\geq 2$, then for any given $\epsilon>0$, there
exists a set $E\in \Sigma$ such that $\mu(E)<\epsilon$ and for almost every
$x$,
$\displaystyle\limsup_{N\to \infty}\dfrac1{N}\displaystyle\sum_{n\in
  [N]}\setone_E(T^{s_n}x)\geq \delta$. The constant $\delta$
depends on $b$ and is explicitly given by $\frac{1}{\zeta(b)}$, where
$\zeta(.) $ is the Riemann zeta function. For all $b\geq 2$, $\delta$ lies
in $ (0,1).$ In this paper, we
will prove a general result, Proposition \ref{prop:1}, as a consequence of which
we will prove the following theorems:
 \begin{thm}[label={thm:5}]{}{} Let $\alpha$ be a fixed 
non-integer rational number. Then for every aperiodic dynamical
system $(X,\Sigma, \mu,T^t)$ and every $\epsilon>0$, there exists a
set $E\in \Sigma$ such that $\mu(E)<\epsilon$, and
$\displaystyle\limsup_{N\to \infty}\dfrac1{N}\displaystyle\sum_{n\in
[N]}\setone_E(T^{n^\alpha}x)=1$ a.e. and
$\displaystyle\liminf_{N\to \infty}\dfrac1{N}\displaystyle\sum_{n\in
[N]}\setone_E(T^{n^\alpha}x)=0$ a.e.
\end{thm}
This result tells us that an extreme case of oscillation occurs for the averages along the sequence $(n^\alpha)$, when
\emph{$\alpha$ is a non-integer rational}.\\
It is interesting to compare this result with \autocite[Theorem B]{BKQW} which says that $\lfloor{n^\frac{3}{2}\rfloor}$ is {pointwise good} for $L^2$. In the same paper, one can find various interesting results about the behavior of averages along sequences which are modeled on functions from the Hardy field.

In our next theorem, we will give an example of a more general class of
sequences which exhibit similar behavior.
 \begin{thm}[label={thm:4}]{}{}Fix a positive integer $l$. Let $\alpha_i=\frac{a_i}{b_i}$, for
$i\in[l]$, be non-integer rational numbers with
gcd$(b_i,b_j)=1$ for $i\neq j$. Let $S=(s_n)$ be the sequence obtained
by rearranging the elements of the set
$\{n_1^{\alpha_1}n_2^{\alpha_2}\dots n_l^{\alpha_l} :n_i\in\setN
\text{ for all } i\in [l]\}$ in an increasing order. Then for every aperiodic dynamical
system $(X,\Sigma, \mu,T^t)$ and every $\epsilon>0$, there exists a
set $E\in \Sigma$ such that $\mu(E)<\epsilon$ and
$\displaystyle\limsup_{N\to \infty}\dfrac1{N}\displaystyle\sum_{n\in
[N]}\setone_E(T^{s_n}x)=1$ a.e. and
$\displaystyle\liminf_{N\to \infty}\dfrac1{N}\displaystyle\sum_{n\in
[N]}\setone_E(T^{s_n}x)=0$ a.e.
 \end{thm}
Theorem \ref{thm:5}
is a special case of Theorem \ref{thm:4}. The sequences considered in Theorem \ref{thm:4} indeed form a much larger class than the ones considered in
Theorem \ref{thm:5}. For example, Theorem \ref{thm:5} does not apply to the sequence
$(m^\frac{1}{3}n^\frac{2}{5}:m,n\in \setN)$, but
Theorem \ref{thm:4} does.\\
In Theorem \ref{thm:conze1} we will prove the Conze principle for flows.
\section{Preliminaries}
\vspace{.3cm}
Let $(X,\Sigma, \mu)$ be a
Lebesgue probability space. That means $\mu$ is a countably additive
complete positive measure on $\Sigma$ with $\mu(X)=1$, with the
further property that the measure space is measure theoretically
isomorphic to a measurable subset of the unit interval with Lebesgue
measure.  By a flow $\{T^t:t \in \setR\}$ we mean a group of
measurable transformations $T^t : X\rightarrow X$ with $T^0(x) = x$,
$T^{t+s} = T^t\circ T^s$, $s,t \in \setR$. The flow is called
measurable if the map $(x,t) \rightarrow T^t(x)$ from $X\times \setR$
into $X$ is measurable with respect to the completion of the product
of $\mu$, the measure on $X$, and the Lebesgue measure on $\setR$. The
flow will be called \emph{measure preserving} if it is measurable and each
$T^t$ satisfies  $\mu ({T^t}^{-1}A)=\mu(A)$ for all
$A\in \Sigma$. The
quadruple $(X,\Sigma,\mu, T^t)$ will be called a dynamical system.
A measurable flow will be called \emph{aperiodic} (free) if
$\mu\{x|T^t(x)=T^s(x) \text{ for some } t\neq s\}=0.$ If
the flow is aperiodic, then the quadruple will be called an aperiodic
system.\\
We will use the following notation to denote the averages.
\begin{equation*}
  \setA_{n\in [N]}f (T^{s_n}x)=\frac{1}{N}\displaystyle\sum_{n\in [N]}f(T^{s_n}x). 
\end{equation*}
More generally, for a subset $J\subset \setN$,
\begin{equation*}
  \setA_{n\in J}f (T^{s_n}x):=\frac{1}{\# J}\displaystyle\sum_{n\in J}f(T^{s_n}x). 
\end{equation*}

Before we go to the proof of our main results, let us give the following definitions.
\begin{defn}[label={defn:2}]{}{} Let $0<\delta\leq 1.$ We say that a sequence $(s_n)$ is
\emph{$\delta$-sweeping out} if in every aperiodic system $(X,\Sigma,
\mu, T^t)$, for a given $\epsilon>0$, there is a set $E\in \Sigma$
with $\mu(E)<\epsilon$ such that $\displaystyle\limsup_{N\to \infty}
\setA_{n\in [N]}\setone_E(T^{s_n}x)\geq
\delta$ for almost every $x\in X$. If $\delta=1$, then $(s_n)$ is
said to be \emph{strong sweeping out}.
\end{defn}
\begin{itemize}
    \item The \emph{relative upper density} of a subsequence $B=(b_n)$ in
$A=(a_n)$ is defined by \\ $\overline{{d}}_A(B):= \displaystyle
\limsup_{N\to \infty }\dfrac{\#\{a_n\in B:n\in[N]\}}{N}$.\\
\item Similarly,
the \emph{relative lower density} of a subsequence $B=(b_n)$ in
$A=(a_n)$ is defined by \\ $\underline{{d}}_A(B):= \displaystyle
\liminf_{N\to \infty }\dfrac{\#\{a_n\in B:n\in[N]\}}{N}$.\\
\end{itemize}
  If both
the limits are equal, we just say \emph{relative density} and denote
it by $d_A(B).$

\begin{remarks}
  \begin{enumerate}[(a)]

  \item If a sequence $(s_n)$ is $\delta$-sweeping out, then by
Fatou's lemma, it is {pointwise bad} for $L^p$,
$p\in [1,\infty].$ If a sequence $(s_n)$ contains a subsequence
$(b_n)$ of relative density $\delta>0$ which is {strong sweeping out}, then $(s_n)$ is $\delta$-sweeping out, and hence {pointwise bad} for $L^p$,
$p\in [1,\infty].$
\item If $(s_n)$ is {strong sweeping out}, then by
  \autocite[Theorem 1.3]{JR}, there exists a residual subset $\Sigma_1$ of $\Sigma$ in the symmetric pseudo-metric with the property that for all $E\in \Sigma_1$,
    \begin{equation}
      \label{eq:rem} \displaystyle\limsup_{N\to \infty}
\setA_{n\in [N]}\setone_E(T^{s_n}x)=1 \text{ a.e. 
and } \displaystyle\liminf_{N\to \infty}
\setA_{n\in [N]}\setone_E(T^{s_n}x)= 0 \text{ a.e.}
    \end{equation}
This, in particular, implies that for every $\epsilon>0$, there exists a set $E\in \Sigma$ such that $\mu(E)<\epsilon$ and $E$ satisfies \eqref{eq:rem}.
  \end{enumerate}
\end{remarks}

\section{Proof of the main results}
\vspace{.3cm}
The proof of Theorem \ref{thm:5}
and Theorem \ref{thm:4} has two main ideas, one is what we call here the grid method, and which has its origins in \cite{JO}, the other one is partitioning a given sequence of real numbers into linearly independent pieces.
Our plan is to prove the theorems as an application of Proposition \ref{prop:1}.  One of the main
ingredients for the proof of Proposition \ref{prop:1} is the following theorem:
\begin{thm}[label={thm:conze}]{}{} Let $(s_n)$ be a sequence of
positive real numbers. Suppose that for any given $\epsilon\in (0,1)$,
$P_0>0$ and a finite constant $C$, there exist $P>P_0$ and a dynamical
system $(\tilde{X},\beta, \m, U^t)$ with a set $\tilde{E}\in \beta$ such that
  \begin{equation}
    \label{eq:denial} \m\Big(x| \max_{P_0\leq N\leq P}
\setA_{n\in [N]}\setone_{\tilde{E}}(U^{s_{n}}x)\geq 1-\epsilon\Big)>
C\m(\tilde{E}).
  \end{equation} Then $(s_n)$ is {strong sweeping out}.
\end{thm}
\begin{proof} This will follow from \autocite[Theorem 2.2]{ABJLRW}
and Proposition \ref{prop:Calderon1} below.
\end{proof}

Let $\lambda$ denote the Lebesgue measure on $\setR$.
\begin{itemize}
\item For a Lebesgue measurable subset $B$ of $\setR$, we define the
\emph{upper density} of $B$ by\\
  \begin{equation} \overline{d}(B)=\displaystyle\limsup_{J\to
\infty}\frac{1}{J} \lambda\big(B\cap [0,J]\big).
  \end{equation}

\item Similarly, for a locally integrable function $\phi$ on $\setR$, we define the \emph{upper
density} of the function as follows:
  \begin{equation} \overline{d}(\phi)=\displaystyle\limsup_{J\to
\infty}\frac{1}{J} \int_{[0,J]}\phi(t)dt.
  \end{equation}

\end{itemize}

\begin{prop}[label={prop:Calderon1}]{}{} Let $(s_n)$ be a sequence of positive real
numbers. Suppose there exists $\epsilon\in (0,1)$, $P_0>0$ and a finite
constant $C$ such that for all $P$ and Lebesgue measurable set $E\subset \setR$
  \begin{equation}
    \label{eq:11maxinR} \overline{d}\Big(t\in \setR| \max_{P_0\leq N\leq P}
\setA_{n\in [N]}\setone_{E}(s_n+t)\geq
1-\epsilon\Big)\leq C\overline{d}(E).
  \end{equation} Then for any dynamical system
$(\tilde{X},{\beta},\m,U^t)$, and for any $\tilde{E}\in \beta$ we
have
  \begin{equation}
    \label{eq:11max in1} \mathfrak{m}\Big(x| \max_{P_0\leq N\leq P}
\setA_{n\in [N]}\setone_{\tilde{E}}(U^{s_{n}}x)\geq 1-\epsilon\Big)\leq
C\m(\tilde{E}).
  \end{equation}
\end{prop}
We will use here a similar argument as Calderon's transference principle \cite{Calderon}.
\begin{proof} 
Let $P>P_0$ be an arbitrary integer. Choose $\eta>0$ small. It will be sufficient to show that for any $\tilde{E}\in \beta$ we
have\\
  \begin{equation}
    \label{eq:12max in1} \mathfrak{m}\Big(x| \max_{P_0\leq N\leq P}
\setA_{n\in [N]}\setone_{\tilde{E}}(U^{s_{n}}x)\geq 1-\epsilon\Big)\leq
(C+\eta)\m(\tilde{E}).
  \end{equation}
To prove \eqref{eq:12max in1}, we need the following lemma:
\begin{lem}{}{}
Under the hypothesis of Proposition \ref{prop:Calderon1}, there exists $J_0=J_0(\eta)\in \setN$ such that for every Lebesgue measurable set $E\subset \setR$ the following holds:\\
\begin{equation}
    \label{eq:131maxinR} \lambda\Big(t\leq J_0| \max_{P_0\leq N\leq P}
\setA_{n\in [N]}\setone_{E}(s_n+t)\geq
1-\epsilon\Big)\leq (C+\eta)\lambda(E\cap [0,J_0]).
  \end{equation}
\end{lem}
\begin{proof}
On the contrary, let us assume that the conclusion is not true. Then for every integer $J_k$, there exists a Lebesgue measurable set $E_k\subset \setR$ such that
\begin{equation}
    \label{eq:231maxinR} \lambda\Big(t\leq J_k| \max_{P_0\leq N\leq P}
\setA_{n\in [N]}\setone_{E_k}(s_n+t)\geq
1-\epsilon\Big)> (C+\eta)\lambda(E_k\cap [0,J_k]).
  \end{equation}
Fix such a $J_k$ and $E_k.$ Observe that the set $E_k':= E_k\cap [0,J_k+S_p]$ also satisfies \ref{eq:231maxinR}. Then for any $J>J_k+S_p$, we have\\
\begin{equation}
    \label{eq:232maxinR} \lambda\Big(t\leq J| \max_{P_0\leq N\leq P}
\setA_{n\in [N]}\setone_{E_k'}(s_n+t)\geq
1-\epsilon\Big)> (C+\eta)\lambda(E_k'\cap [0,J]).
  \end{equation}
Letting $J\to \infty$, we get
\begin{equation}
    \label{eq:233maxinR} \overline{d}\Big(t\in \setR: \max_{P_0\leq N\leq P}
\setA_{n\in [N]}\setone_{E_k'}(s_n+t)\geq
1-\epsilon\Big)> (C+\eta)\overline{d}(E_k').
  \end{equation}
  But this is a contradiction to our hypothesis \eqref{eq:11maxinR}. This finishes the proof of our lemma.
\end{proof}
Now we will establish \eqref{eq:12max in1}.
Let $\tilde{E}$ be an arbitrary element of ${\beta}$.
For a fixed element $x\in \tilde{X}$, define $\setone_{{E_x}}$ by
  \begin{equation}\label{eq:pullback} \setone_{{E_x}}(t):= \setone_{\tilde{E}}{(U^{t}
x)}
    \end{equation} Applying  \eqref{eq:131maxinR}
on the set $E_x$ we get\\
  \begin{equation}
    \label{eq:132maxinR} \lambda\Big(t\leq J_0| \max_{P_0\leq N\leq P}
\setA_{n\in [N]}\setone_{E_x}(s_n+t)\geq
1-\epsilon\Big)\leq (C+\eta)\lambda(E_x\cap [0,J_0]).
  \end{equation} By definition of the sets $E_x$ and $\tilde{E}$,  we get\\
  \begin{equation*} \max_{P_0 \leq N\leq
P}\setA_{n\in [N]}\setone_{\tilde{E}}(U^{s_{n}+t}x)=\max_{P_0 \leq N\leq
P}\setA_{n\in [N]}\setone_{E_x}({s_{n}}+t), \text{ for all } t.
  \end{equation*}  Substituting the
above equality in  \eqref{eq:132maxinR}, we get
  \begin{equation}
    \label{eq:max inR3} \lambda\Big(t|\ t\leq J_0,\ \max_{P_0 \leq
N\leq P}\displaystyle\setA_{n\in [N]}\setone_{\tilde{E}}(U^{s_{n}+t}x)\geq 1-\epsilon\Big)\leq
(C+\eta)\lambda(E_x\cap [0,J_0]).
  \end{equation} Introducing the set\\
  \begin{equation} B:=\Big\{y: \max_{P_0 \leq N\leq P}\displaystyle\setA_{n\in [N]}\setone_{\tilde{E}}(U^{s_n} y)\geq
1-\epsilon\Big\}
  \end{equation} we can rewrite equation  \eqref{eq:max inR3} as
  
    \begin{align*} &\lambda\big(t: t\leq J_0 \text{ and }
\setone_B(U^t x)=1\big)\leq (C+\eta)\lambda(E_x\cap[0,J_0])\\ \intertext{ which implies}
&\int_{\setR} \setone_{\{t\leq J_0\}}(t)\setone_B(U^t x)dt\leq
(C+\eta)\int_{[0,J_0]}\setone_{E_x}(t)dt
    \end{align*}

  Integrating the above equation with respect to the $x$ variable, using
 \eqref{eq:pullback} in the right hand side and applying Fubini's
theorem we get
  \begin{equation} \int_{\setR}\int_{\tilde{X}} \setone_B(U^t
x)d\m(x)\setone_{\{t\leq J_0\}}(t)dt\leq
(C+\eta)\int_{[0,J_0]}\int_{\tilde{X}}\setone_{\tilde{E}}(U^t x)d\m(x)dt
  \end{equation} By using the fact that $U^t$ is measure preserving we
get\\
  \begin{equation}
    \label{eq:131} \m (B) J_0 \leq (C+\eta) \m(\tilde{E}) J_0.
  \end{equation} This gives us that \\
  \begin{equation} \m(B)\leq (C+\eta) \m (\tilde{E}),
  \end{equation} finishing the proof of the desired maximal inequality
 \eqref{eq:12max in1}. This completes the proof.
\end{proof}

One can generalize the above proposition as follows:
\begin{prop} [label={prop:Calderon2}]{}{} Let $(s_n)$ be a sequence of
positive real numbers. Suppose there exists a finite constant $C$ such
that for all locally integrable functions $\phi$ and $\gamma>0$ we have
  \begin{equation}
    \label{eq:1maxinR} \overline{d}\Big(t| \sup_{N}
\displaystyle\setA_{n\in [N]}\phi(s_n+t)\geq \gamma\Big)\leq
\frac{C}{\gamma}\overline{d}(|\phi|).
  \end{equation} Then for any dynamical system
$(\tilde{X},{\beta},\m,U^t)$, and for any $\tilde{f}\in L^1$ we have\\
  \begin{equation}
    \label{eq:max in1} \m\Big(x| \sup_{ N}
\displaystyle\setA_{n\in [N]}{\tilde{f}}(U^{s_{n}}x)\geq
\gamma \Big)\leq \frac{C}{\gamma}\int_{\tilde{X}}|\tilde{f}|d\m.
  \end{equation}
\end{prop} The proof of this proposition is very similar to the proof
of Proposition \ref{prop:Calderon1}, and hence we will skip the proof.

The underlying principle of Theorem \ref{thm:conze} is the Conze principle
which has been widely used to prove many results related to the Birkhoff's pointwise ergodic theorem.  Originally, the theorem was proved for a single
transformation system by Conze \autocite{conze}. Here we will prove a version
of the theorem for flows. The theorem is not required for proving
Proposition \ref{prop:1}, but it
is of independent interest.

\begin{thm}[label={thm:conze1}]{Conze principle for flows}{} Let $S=
(s_n)$ be a sequence of positive real numbers and\\ $(X,\Sigma,\mu,
T^t)$ be an aperiodic flow which satisfies the following maximal
inequality:

  There exist a finite constant $C$ such that for all ${f}\in L^1(X)$
and $\gamma>0$ we have

  \begin{equation}
    \label{eq:max in} \mu\Big(x| \sup_{N}
\displaystyle\setA_{n\in [N]}{{f}}(T^{s_{n}}x)\geq
\gamma\Big)\leq \frac{C}{\gamma}\int_{X} |f|d\mu.
  \end{equation}

  Then the above maximal inequality holds in every dynamical system with the same constant $C$.
 
\end{thm}

For a Lebesgue measurable subset $A$ of $\setR$, and $F\in \Sigma$, we
will use the notation $T^A(F)$ to denote the set $\cup_{t\in
A}T^t(F)$. We will say that $T^A(F)$ is disjoint if we have
$T^t(F)\cap T^s(F)=\emptyset$ for all $t\neq s.$
\begin{proof} We want to show that the maximal inequality
 \eqref{eq:max in} transfers to $\setR$. If we can show that, then
Proposition \ref{prop:Calderon2} will finish the proof.\\ Let $\phi$ be a
measurable function. After a standard reduction, we can assume that
$\gamma=1$ and $\phi$ is positive. It will be sufficient to show that
for any fixed $P$ we have
  \begin{equation}
    \label{eq:max inR5} \overline{d}\Big(t| \max_{N\leq P}
\displaystyle\setA_{n\in [N]}\phi(s_n+t)\geq 1\Big)\leq
C\overline{d} (\phi).
  \end{equation} Define
  \begin{equation} G:=\Big\{t: \max_{N\leq
P}\setA_{n\in [N]}\phi({s_{n}}+t)>1\Big\}
  \end{equation} Let $(L_n)$ be a sequence such that\\
  \begin{equation}
    \label{eq:limsup1} \overline{d}(G)=\lim_{n\to \infty}
\frac{1}{L_n-s_P}\lambda\big([0,L_n-s_P]\cap G\big)
\end{equation} Fix $L_n$ large enough so that it is bigger than $s_P$.
Then by applying the $\setR$-action version of Rohlin tower \autocite[Theorem 1]{Lind}, we can find a set $F\subset X$, such that
$T^{[0,L_n]}(F)$ is disjoint, measurable and
$\mu(T^{[0,L_n]}F)>1-\eta$, where $\eta>0$ is small. It also has the
additional property that if we define the map $\nu$ by
$\nu(A):=\mu(T^A(F))$ for any Lebesgue measurable subset $A$ of $[0,L_n]$,
then $\nu$ becomes a constant multiple of Lebesgue measure.\\ Define
$f :X\to \setR$ as follows:
\begin{align*} f(x)=
      \begin{cases} &\phi(t) \text{ when } x\in T^t(F)\\ &0 \text{
when } x\not\in T^{[0,L_n]}(F).
      \end{cases}
    \end{align*}
  Now we observe that for each $x\in T^{t}(F)$ with
$t\in[0,L_n-s_P]$ and $N\leq P$ we have
  \begin{equation*} \max_{N\leq P}\frac{1}{N}\displaystyle\sum_{n\leq
N}\phi({s_{n}}+t)=\max_{N\leq P}\frac{1}{N}\displaystyle\sum_{n\leq
N}f(T^{s_{n}}x)
  \end{equation*} Substituting this into  \eqref{eq:max in}, we get
  \begin{equation}
    \label{eq:max in6} \mu\Big(x|\ x\in T^{t}(F) \text{ for some }
t\in [0,L_n-s_P]\cap G\Big)\leq C\int_X fd\mu.
  \end{equation} Then  \eqref{eq:max in6} can be rewritten as\\
  \begin{equation}
    \label{eq:rokhlin1} \mu\big(T^{[0,L_n-s_P]\cap G} (F)\big)\leq
C\int_X fd\mu.
  \end{equation}

  Applying the change of variable formula, we can rewrite
 \eqref{eq:rokhlin1} as
    \begin{align*}
    \nu\big([0,L_n-s_P]\cap G\big)&\leq C\int_{[0,L_n]}
\phi(t) d\nu(t). \intertext{Since, $\nu$ is a constant multiple of the Lebesgue measure, we get} 
\lambda\big([0,L_n-s_P]\cap G\big)&\leq
C\int_{[0,L_n]} \phi(t) dt .
\intertext{Dividing both sides by $(L_n-s_P)$ and taking liminf and limsup, we have} \displaystyle\liminf_{n\to
\infty} \frac{1}{L_n-s_P}\lambda\big([0,L_n-s_P]\cap G\big)&\leq C
\displaystyle\limsup_{n\to \infty} \frac{L_n}{L_n-s_P}\frac{1}{L_n}\int_{[0,L_n]} \phi(t) dt \intertext{This implies, by
 \eqref{eq:limsup1}, that}
\overline{d}(G)&\leq C\overline{d}(\phi).
    \end{align*}
    This is the desired maximal inequality on $\setR.$
Now we will invoke Proposition \ref{prop:Calderon2} to finish the proof.
\end{proof} We need the following two lemmas to prove Proposition \ref{prop:1}.
\begin{lem}[label={lem:1}]{}{} Let $S$ be a finite subset of $\setR$
such that $S$ is linearly independent over $\setQ$. Suppose ${S}=
\cdot\hspace{-6.5pt}\bigcup_{q\leq Q} {S}_{q}$, be a partition of $S$
into $Q$ sets. Then there is a positive integer $r$ such that for every
$q\in [Q]$,
  \begin{equation}
    \label{eq:kronecker} rs\in I_q \text{ (mod } 1) \text{ whenever }
s\in {S}_{q},
  \end{equation} where
  \begin{equation}
    \label{eq:kronecker'} I_q:= \Big(\frac{q-1}{Q},\frac{q}{Q}\Big).
  \end{equation}
\end{lem}
\begin{proof}
 The above lemma is a consequence of the well known Kronecker's diophantine theorem.
\end{proof}
Let us introduce the following notation for stating our
next lemma.  Let $\mathcal{P}$ denote the set of prime numbers.  For a
fixed positive integer $m$, we want to consider the linear
independence of the set consisting of products of the form
$\prod_{p\in\mathcal{P}}p^{\frac{v_p}{m}}$ over $\setQ$, where only
finitely many of the exponents $v_p$ are nonzero.  Let $S$ be a
possibly infinite collection of such products, that is,
$S=
\big\{\prod_{p\in\mathcal{P}} p^{\frac{v_{j,p}}{m}}:j\in J\big\}$,
 where $J$ is a (countable) index set.  We call $S$ a
\emph{good} set if it satisfies the following two conditions:
\begin{enumerate}[(i)]
\item For any $j\in J$ and $p\in \mathcal{P}$, if $v_{j,p}\neq 0$,
then $v_{j,p}\not\equiv 0$ (mod $m$).
\item The vectors $\mathbf v_j= ({v_{j,p}})_{p\in\mathcal{P}}$, $j\in
J$, of exponents are \emph{different} (mod $m$). This just means that
if $i\ne j$, then there is a $p\in\mathcal{P}$ so that
$v_{j,p}\not\equiv v_{i,p}\pmod m$.
\end{enumerate}

\begin{lem}[label={lem:2}]{Besicovitch reformulation}{} Let $m$ be a
positive integer, and let the set $S$ be {good}. Then $S$ is a linearly independent set over the rationals.
  
\end{lem}

\begin{proof} Suppose, if possible, there exist
  $\lambda_i\in \setQ, \text { for } \ i\in [N]$ such that
  $\sum\limits_{i\leq N}\lambda_i s_i=0$ be a non-trivial relation in
  $S$. We can express this relation as
  $P(p_1^\frac{1}{m},p_2^\frac{1}{m},\dots,p_r^\frac{1}{m})=0$, where
  $P$ is a polynomial, and for all $i\in [r]$, $p_i$ is a prime
  divisor of $(s_j)^m$ for some $j\in [N].$ Since $S$ is a \emph{good} set,
  we can reduce this relation into a relation
  $P'(p_1^\frac{1}{m},p_2^\frac{1}{m},\dots,p_r^\frac{1}{m})=0$ such that
  each coefficient of $P'$ is a non-zero constant multiple of the
  corresponding coefficient of $P$, and the degree of $P'$ with
  respect to $p_i^\frac{1}{m}$ is $<m$ for all $i\in[r].$ But this
  implies, by \autocite[Corollary 1]{BE}, that all coefficients of $P'$
  vanish. Hence all the coefficients of $P$ will also vanish. But this
  is a contradiction to the non-triviality of the relation that we
  started with. This completes the proof.
\end{proof}
\begin{prop}[label={prop:1}]{}{} Let $S=(s_n)$ be an increasing sequence of positive real\\ numbers. Suppose $S$ can be written as a disjoint union of $(S_k)_{k=1}^\infty$ i.e. $S=\cdot\hspace{-11pt}\bigcup\limits_{k\in \setN}S_{k}$ such that the following two conditions hold:\\ \begin{enumerate}[(a)]
  \item \emph{(Density condition):} The relative upper density of $R_K
:= \bigcup\limits _{k\in[K]}S_k$ goes to 1 as $K\to \infty$.
  \item \emph{(Linear independence condition):} Each subsequence $S_k$
of $S$ is linearly independent over $\setQ$.
  \end{enumerate}

  Then $S$ is {strong sweeping out}.
\end{prop}

\begin{proof} We will apply Theorem \ref{thm:conze}
to prove this result.  So, let $\epsilon\in (0,1), P_0$ and $C>0$ be
arbitrary.  Choose $\rho>1$ such that $\delta:={\rho}{(1-\epsilon)}<1$.
By using condition (a), we can choose $K$ large enough so that
$\overline{d}(R_K)> \delta.$\\ This implies that we can find a
subsequence $(s_{n_j})$ of $(s_n)$ and $N_0\in \setN$ such that
  \begin{equation}
    \label{eq:205} \dfrac{\#\big\{s_n\in R_K:n\leq
n_j\big\}}{n_j}>\delta \text{ for all } j\geq N_0.
  \end{equation} Define $C_K:=S\setminus R_K=\bigcup\limits_{j>K}S_j$.
 
  Let $P$ be a large number which will be determined later. For an
interval $J\subset [P]$ of indices, we consider the average
$\displaystyle\setA_{n\in J}f(U^{s_n}x)$, where $U^t$ is a measure preserving flow on a space
which will be a torus with high enough dimensions for our purposes.
 
  We let $\tilde{S}$ be the truncation of the given sequence $S$ up to
the $P$-th term. Define $\tilde{S_k}$ and $\tilde{C_K}$ be the
corresponding sections of $S_k$ and $C_K$ in $\tilde{S}$ respectively,
i.e.
 \begin{align*} \tilde{S}:= (s_n)_{n\in [P]},&\ \ \ \tilde{S_k}:=
\tilde{S}\cap {S_k},\ \text{for }k\in[K] \\ &\tilde{C_K}:=
\tilde{S}\cap C_K.
    \end{align*}

  The partition $\tilde{S}_1, \tilde{S}_2,\dots,\tilde{S}_K,
\tilde{C_K}$ of $\tilde{S}$ naturally induces a partition of the index
set [P] into $K+1$ set $\mathcal{N}_k, k\leq K+1$ where $\tilde{S}_k$
corresponds to $\mathcal{N}_k$ for $k\in[K]$ and $\tilde{C_K}$
corresponds to $\mathcal{N}_{K+1}.$  The averages $\displaystyle\setA_{n\in J}f(U^{s_n}x)$ can be written as
\begin{equation}
    \label{eq:2} \displaystyle\setA_{n\in J}f(U^{s_n}x)= \frac{1}{\# J}\sum_{k\leq {K+1}}\sum_{n\in
J\cap \mathcal{N}_{k}}f(U^{s_n} {x}).
  \end{equation}

  By hypothesis (b) and Lemma \ref{lem:1}, it follows that each
$\tilde{S}_k$ has the property that if it is partitioned into $Q$
sets, so $\tilde{S}_k= \cup_{q\leq Q} \tilde{S}_{k,q}$, then there is
an integer $r$ so that \\
 
  \begin{equation}
    \label{eq:4} rs\in I_q \text{ mod } 1 \text{ for } s\in
\tilde{S}_{k,q} \text{ and } q\leq Q.
  \end{equation} where \begin{equation}
    \label{eq:5} I_q:= (\frac{q-1}{Q},\frac{q}{Q}).
  \end{equation} The space of action is $K$ dimensional torus
$\setT^K$, subdivided into little $K$ dimensional cubes $C$ of the
form \begin{equation}
    \label{eq:6} C= I_{q(1)}\times I_{q(2)}\times\dots \times I_{q(k)}
\text{ for some } q(k)\leq Q \text{ for } k\leq K.
  \end{equation} At this point, it is useful to introduce the
following vectorial notation to describe these cubes $C$. For a vector
$\mathsf{q}=\big(q(1),q(2),\dots,q(K)\big)$, with $q(k)\leq Q$,
define \begin{equation}
    \label{eq:7} I_{\mathsf{q}}:= I_{q(1)}\times I_{q(2)}\times \dots
\times I_{q(K)}.
  \end{equation}
Since, each component $q(k)$ can take up the values 1,2,\dots ,Q we
divided $\setT^K$ into $Q^K$ cubes.\\ We also consider the ``bad'' set
$E$ defined by \begin{equation}
    \label{eq:8} E:= \displaystyle\bigcup_{k\leq K}
(0,1)\times(0,1)\times\dots \times \underbrace{\big(I_1\cup
I_2\big)}_\text{k-th coordinate}\times\dots \times (0,1).
  \end{equation}

  Defining the set $E_k$ by
  \begin{equation}
    \label{eq:9} E_k:= (0,1)\times (0,1) \times \dots \times
\underbrace{\big(I_1\cup I_2\big)}_\text{k-th coordinate}\times \dots
\times (0,1)
  \end{equation} we have \\
  \begin{equation}
    \label{eq:10} E=\displaystyle\bigcup_{k\leq K}E_k \text{ and }
\lambda^{(K)}(E_k)\leq \frac{2}{Q} \text{ for every } k\leq K
  \end{equation} where $\lambda^{(K)}$ is the Haar-Lebesgue measure on
$\setT^K$.

  By  \eqref{eq:10}, we have
  \begin{equation}
    \label{eq:11} \lambda^K(E)\leq \frac{2K}{Q}.
  \end{equation}

  Observe that by taking $Q$ very large, one can make
  sure that the measure of the bad set
  \begin{equation}
    \label{eq:badset} \lambda^{(K)}\big(E\big)<\frac{1}{C}.
  \end{equation}
  
  Now the idea is to have averages that move each of these little
  cubes into the supports of the set E.  The $2$-dimensional version
  of the process is illustrated in Figure \ref{fig:grid}.
    \begin{figure}[h!]

    \begin{tikzpicture}[scale=.75]
		
  \tkzDefPoints{0/0/O, 10/10/M}
  
  \tkzDefPoints{0/2/f2, 0/4/f4, 0/6/f6, 0/7/f7, 0/8/f8, 0/10/f10,
    2/0/h2, 4/0/h4, 5/0/h5, 6/0/h6, 8/0/h8, 10/0/h10}

  \tkzDefPoints{4/10/p4, 5/10/p5, 10/7/p7, 10/8/p8}
  \tkzDefPoints{10/2/g102, 2/10/g210}
  \tkzDefPoints{1.5/1.5/es, 6.5/3.5/bs}
  \tkzDefPoints{-1/-1/ee, 11.5/3.5/be}

  \tkzLabelPoint[left](O){$0$}
  \tkzLabelPoint[left](f2){$\frac2{10}$}
  \tkzLabelPoint[left](f4){$\frac4{10}$}
  \tkzLabelPoint[left](f6){$\frac6{10}$}
  \tkzLabelPoint[left](f7){$\frac7{10}$}
  \tkzLabelPoint[left](f8){$\frac8{10}$}

  \tkzLabelPoint[below](h2){$\frac2{10}$}
  \tkzLabelPoint[below](h4){$\frac4{10}$}
  \tkzLabelPoint[below](h5){$\frac5{10}$}
  \tkzLabelPoint[below](h6){$\frac6{10}$}
  \tkzLabelPoint[below](h8){$\frac8{10}$}

  \tkzDrawPolygon[fill=orange!80](O,h10,g102,f2)
  \tkzDrawPolygon[fill=orange!80](O,f10,g210,h2)
  \tkzDrawPolygon[fill=blue!60](h4,h5,p5,p4)
  \tkzDrawPolygon[fill=blue!60](f7,f8,p8,p7)
  \draw (O) grid (M);

  \tkzDrawSegments[->](es,ee bs,be)
  \tkzLabelPoint[below](ee){bad set $E$}
  \tkzLabelPoint[right](be){$(x_1,x_2)\in B_{6,3}$}
  \tkzDrawPoint[red](bs)
\end{tikzpicture}
    \caption[-5cm]{Illustration of the 2-dimensional case.
      Here the ``bad set'' E is the orange colored region. Let
      $(x_1,x_2)$ be an arbitrary point (which belongs to $B_{6,3}$ in
      this case). We need to look at an average where
      $r_1s_n\in (\frac{4}{10},\frac{5}{10})$ for all
      $n\in \tilde{S_1}$ and $r_2s_n\in (\frac{7}{10},\frac{8}{10})$
      for all $n\in \tilde{S_2}.$ Then it would give us
      $(x_1,x_2)+(r_1s_n,r_2s_n)\in E$ for all
      $n\in \tilde{S_1}\cup \tilde{S_2}.$ The picture suggests the
      name `grid method' }
    \label{fig:grid}
  \end{figure}
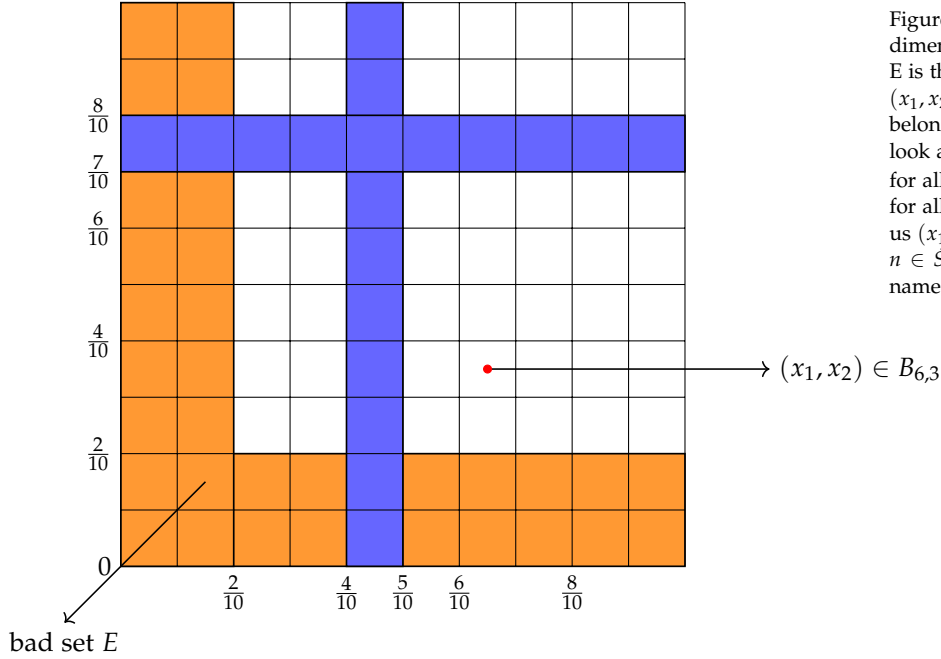

  Since we have
  $Q^K$ little cubes, we need to have $Q^K$ averages $\setA_J$. This means
  we need to have $Q^K$ disjoint intervals $J_i$ of indices. The
  length of these intervals $J_i$ needs to be ``significant''. More
  precisely, we will
  choose $J_i$ so that it satisfies the following two conditions:
  \begin{enumerate}[(i)]
  \item \begin{equation}
      \label{eq:interval} \ J_0=[1,P_0],
\frac{\#J_i}{\sum\limits_{0\leq j\leq i}\#J_{j}}>\frac{1}{\rho} \text
{ for all } i\in [Q^K].
    \end{equation}
  \item If $J_i=[N_{i},N_{i+1}],$ then $(N_{i+1}-N_{i})$ is an element
of the subsequence $(n_j)$ and
  \begin{equation}\label{eq:149}
  (N_{i+1}-N_{i})>N_0 \text{ for all } i\in [Q^K].
  \end{equation}
  \end{enumerate}
 So we have $Q^K$ little cubes $C_i, i\leq Q^K$ and $Q^K$ intervals
$J_i$, $i\leq Q^K$. We match $C_i$ with $J_i$. The large number $P$
can be taken to be the end-point of the last intervals $J_{Q^K}.$\\

  We know that each $C_i$ is of the form \\
  \begin{equation}
    \label{eq:112} C_i= I_{\mathsf{q}_i},
  \end{equation} for some $K$ dimensional vector $\mathsf{q}_i=
\big(q_i(1),q_i(2),\dots q_i(K)\big) $ with $q_i(k)\leq Q$ for every
$k\leq K.$ The interval $J_i$ is partitioned as\\
  \begin{equation}
    \label{eq:13} J_i= \displaystyle\bigcup_{k\leq K+1}(J_i\cap
\mathcal{N}_k).
  \end{equation} For a given $k\leq K$, let us define the set of
indices $\mathcal{N}_{k,q}$ for $q\leq Q$, by \\
  \begin{equation}
    \label{eq:14} \mathcal{N}_{k,q} := \bigcup\limits_{i\leq Q^K,
q_i(k)=q} (J_i\cap \mathcal{N}_k).
  \end{equation} Since, the sets $\tilde{S}_{k,q} := \{s_n | n\in
\mathcal{N}_{k,q}\}$ form a partition of $\tilde{S_k}$, by the
argument of  \eqref{eq:4}, we can find $r_k$ so that \\
  \begin{equation}
    \label{eq:15} r_ks_n\in I_{Q-q} \text{ (mod }1) \text{ for } n\in
\mathcal{N}_{k,q} \text{ and } q\leq Q.
  \end{equation}

  Define the flow $U^t$ on the $K$ dimensional torus $\setT^K$ by \\
  \begin{equation}
    \label{eq:16} U^t(x_1,x_2,\dots, x_K):= (x_1+r_1t,x_2+r_2t,\dots ,
x_K+r_Kt).
  \end{equation} We claim that \\
  \begin{equation}
    \label{eq:171} \Big\{\mathbf{x}|\max_{i\leq Q^K} \displaystyle\setA_{n\in J_i}\setone
_E(U^{s_n}\mathbf{x}) \geq \delta\Big\}= \setT ^K \text{ where }
\textbf{x}=(x_1,x_2,\dots,x_K).
  \end{equation} This will imply that
  \begin{equation}
    \label{eq:17} \Big\{\mathbf{x}|\max_{P_0\leq N\leq P}
\displaystyle\setA_{n\in[N]}\setone
_E(U^{s_n}\mathbf{x})\geq
1-\epsilon\Big\}= \setT ^K.
  \end{equation} To see this implication, let $\mathbf{x}\in \setT^K$
and $\setA_{n\in J_i} \setone _E(U^{s_n}\mathbf{x})\geq \delta$.\\ Assume that
$J_i=[N_i,N_{i+1}]$.\\ Then
    \begin{align*} \setA_{n\in[N_{i+1}]}\setone_E
(U^{s_n}\mathbf{x}) =& \frac{N_i}{N_{i+1}}\frac{1}{N_i}\sum_{n\leq
N_i}\setone_E (U^{s_n}\mathbf{x})+
\frac{N_{i+1}-N_i}{N_{i+1}}\frac{1}{N_{i+1}-N_i}\sum_{n\in
(N_i,N_{i+1}]}\setone_E (U^{s_n}\mathbf{x})
\intertext{ since $ \setA_{n\in J_i} \setone _E(U^{s_n}\mathbf{x})\geq \delta$,} \geq &
\frac{N_{i+1}-N_i}{N_{i+1}}\delta \ \ \ \  \\\intertext{ by  \eqref{eq:interval} } \geq& \frac{\delta}{\rho}\ \ \
\\ = & (1-\epsilon)
    \end{align*} This proves  \eqref{eq:17}.\\ Now let us prove our
claim  \eqref{eq:171}.\\ Let $\mathbf{x}\in C_i$ and consider the
average $\setA_{n\in J_i} \setone _E(U^{s_n}\mathbf{x})$. \\ We have
\begin{align*} \setA_{n\in J_i} \setone _E(U^{s_n}\mathbf{x}) &=\frac{1}{\#
J_i}\sum_{k\leq {K+1}}\sum_{n\in J_i\cap
\mathcal{N}_{k}}\setone_E(U^{s_n} {\mathbf{x}}) \intertext{by  \eqref{eq:16}} &=\frac{1}{\#
J_i}\sum_{k\leq {K+1}}\sum_{n\in J_i\cap
\mathcal{N}_{k}}\setone_E(x_1+r_1s_n,x_2+r_2s_n,\dots ,x_K+r_Ks_n).
    \end{align*}
If we can prove that for each $k\in[K]$,
  \begin{equation}
    \label{eq:20} (x_1+r_1s_n,x_2+r_2s_n,\dots, x_k+r_ks_n,\dots,
x_K+r_Ks_n)\in E_k \text{ if }n\in J_i\cap \mathcal{N}_k
  \end{equation} then we will have\\
  \begin{align*} \setA_{n\in J_i} \setone _E(U^{s_n}\mathbf{x})&\geq \frac{1}{\#
J_i}\sum_{k\leq {K}} \#\big(J_i\cap \mathcal{N}_{k}\big) \intertext{Using the notation $J_i=[N_i,N_{i+1}]$,} &=
\dfrac{\#\{ s_n\in R_K: N_i\leq n \leq N_{i+1}\}}{N_{i+1}-N_i}\ \ \
\\ &\geq \delta
    \end{align*} proving  \eqref{eq:171}.\\ For the last inequality, we used  \eqref{eq:205} and \eqref{eq:149}.\\ We are
yet to prove  \eqref{eq:20}.\\ Since $\mathbf{x}\in C_i= I_{\mathsf{q}}$
we have $x_k\in I_{q_i(k)}$ for every $k$. By the definition of $r_k$
in  \eqref{eq:15} we have $r_ks_n\in I_{Q-q_i(k)}$ if $n\in J_i\cap
\mathcal{N}_k$.\\ It follows that \\
  \begin{equation}
    \label{eq:22} x_k+r_ks_n\in I_{q_{i(k)}}+I_{Q-q_i(k)} \text{ if }
n\in J_i\cap\mathcal{N}_k.
  \end{equation} Since $I_{q_{i(k)}}+I_{Q-q_{i(k)}}\subset I_1\cup
I_2$, we get \\
  \begin{equation}
    \label{eq:23} x_k+r_ks_n\in I_1\cup I_2 \text{ if } n\in J_i\cap
\mathcal{N}_k.
  \end{equation} By the definition of $E_k$ in  \eqref{eq:9}, this
implies that \\
  \begin{equation}
    \label{eq:24} (x_1+r_1s_n,x_2+r_2s_n,\dots x_k+r_ks_n,\dots,
x_K+r_Ks_n)\in E_k
  \end{equation} as claimed.\\

  Thus, what we have so far is the following:

  for every $C>0,\epsilon \in (0,1)$ and $P_0\in \setN$, there exists
an integer $P>P_0$ and a set $E$ in the dynamical system
$(\setT^K,\Sigma^{(K)},\lambda^{(K)},U^t)$ such that\\
  \begin{equation*}
    \label{eq:125} \lambda^{(K)}\Big(\max_{P_0\geq N\geq P}
\setA_{n\in [N]}\setone_E(U^{s_n}\mathbf{x})\geq 1-\epsilon
\Big)=1\geq C\lambda^{(K)}(E).
  \end{equation*} The right hand side of the above inequality is true by
 \eqref{eq:badset}.\\ Now Theorem \ref{thm:conze} finishes the proof.
\end{proof}
We need the following definition to prove Theorem \ref{thm:5}.
\begin{defn}{}{}
Let $b$ be a
fixed integer $\geq 2$. A positive integer $n$ is said to be
\emph{$b$-free} if $n$ is not divisible by $b$-th power of any integer
>1.
\end{defn}

\begin{proof}[Proof of Theorem \ref{thm:5}] If we can prove that the sequence
$(n^\alpha)$ is {strong sweeping out}, then part (b) of remark (in page 4) will give us the desired result.\\ To prove that the sequence
$(n^\alpha)$
{strong sweeping out}, we want to apply Proposition \ref{prop:1} with
$S=(n^{\alpha})$.\\ Observe that the set $S_0:=\{s\in S: s \text{ is
an integer}\}$ has relative density $0$ in the set $S$. Hence we can
safely delete $S_0$ from $S$ and rename the modified set as $S$.  Let
$\alpha=\frac{a}{b} \text{ with gcd}(a,b)=1 \text{ and } b\geq 2$.\\
By applying the fundamental theorem of arithmetic, one can easily prove
that every positive integer $n$ can be uniquely written as
$n={j^b}\tilde{n}$, where $j,\tilde{n}\in \setN$, $\tilde{n}$ is
\emph{$b$-free}.  This observation suggests us that we can partition
the set $S$ as follows:
\begin{lem}{}{}
 $S=\cdot\hspace{-6pt}\cup_{k\in \setN}S_{k}$, where we define
  \begin{equation}
    \label{eq:29} S _1:= \{n^\frac{a}{b}: n \text{ is }
\emph{$b$-free}\} \text{ and } S_k:=k^aS_1=\{k^a n^\frac{a}{b}:n
\text{ is } \emph{$b$-free}\}.
  \end{equation}
\end{lem}
\begin{proof}
 It can be easily seen that $S=\cup_{k\in
\setN}S_{k}.$\\ To check the disjointness let, if possible, there
exists $s\in S_{k_1}\cap S_{k_2}, k_1\neq k_2$. Then there exist
\emph{$b$-free} integers $n_1,n_2$ such that\\
  \begin{equation}
    \label{eq:26} s=k{_i}^a{n_i^\frac{a}{b}}, i=1,2.
  \end{equation} Without loss of any generality, we can assume
gcd$(k_1,k_2)=1$.\\  \eqref{eq:26} implies that \\
  \begin{equation}
    \label{eq:30} k_1^{b}n_1=k_2^{b}n_2.
  \end{equation} gcd$(k_1,k_2)=1\implies \text{
gcd}(k_1^b,k_2^b)=1$.\\ Hence,  \eqref{eq:30} can hold only if
$\mathrel{k_1^b|n_2}$ and $\mathrel{k_2^b|n_1}.$ Since $n_2$ is
\emph{$b$-free}, $k_1^b$ can divide $n_2$ only when $k_1=1.$
Similarly, $k_2^b$ can divide $n_1$ only when $k_2=1.$ But this
contradicts our hypothesis $k_1 \neq k_2$.
\end{proof} 
  We will be done if we can show that each $S_{k}$ satisfies the
hypothesis of Proposition \ref{prop:1}.  We need to show that the upper density of $R_K := \bigcup\limits
_{k\in[K]}S_k$ goes to 1 as $K\to \infty$.\\ In fact, we will show
that, the lower density of $R_K := \bigcup\limits _{k\in[K]}S_k$ goes
to 1 as $K\to \infty$.  Let us look at the complement $C_{K}:=
S\setminus R_{K}.$ It will be enough to show that the upper density of
$C_{K}$ goes to $0$ as $K\to \infty.$\\ An element $n^\alpha$ of $S$
belongs to $C_{K}$ only when $n^{\alpha}\in k^a S$ for some
$k>K$. This means that there is $k>K$ so that $k^b| n$.\\ Hence, we
have\\
  \begin{equation}
    \label{eq:37} C_K\subset \bigcup\limits_{k>K}k^a
S=\bigcup\limits_{k>K}k^a S_1=\bigcup\limits_{k>K} S_k.
  \end{equation} To see that the first equality of  \eqref{eq:37} is
true, let $ x\in \bigcup\limits_{k>K}k^a S $. Then we must have that
$x=k_1^a n^\alpha$ for some $k_1>K.$ Write $n=j^b \tilde{n}$ where
$\tilde{n}$ is $b$-free. Then $x=(jk_1)^a\tilde{n}^\alpha.$ Since,
$jk_1>K$, $x\in \bigcup\limits_{k>K}k^a S_1.$ Other containment is
obvious. We want to show that the upper density of
$A_K:=\bigcup\limits_{k>K}S_k$ goes to $0$ as $K\to \infty.$ This
would imply that the upper density of $C_K$ also goes to $0$, implying
that the lower density of its complement $R_K$ goes to 1. We will see
that we even have a rate of convergence of these lower and upper
densities in terms of $K$.\\ Let us see why the upper density of $A_K$
goes to $0$ as $K\to \infty$. \\ We have,
  \begin{align*} \#S(N)&:=\{n^\alpha\leq N:n\in \setN\}\\
&=\lfloor{N^\frac{a}{b}}\rfloor\\ \#S_k(N)&:=\{(k^bn)^\alpha\leq N:n\in \setN\}\\
&=\big\lfloor{\frac{N^\frac{a}{b}}{k^b}}\big\rfloor
  \end{align*} Hence $\dfrac{\#S_k(N)}{\#S(N)}\leq \frac{1}{k^b}$.\\
We now have, 
\begin{align*}
\dfrac{\#A_K(N)}{\#S(N)}&=\dfrac{\#\Big(\bigcup\limits_{k>K}
S_k(N)\Big)}{\#S(N)}\\ &\leq
\dfrac{\sum\limits_{k>K}\#S_k(N)}{\#S_N}\\ &\leq
\sum\limits_{k>K}\dfrac{1}{k^b}
    \end{align*} Since, $\sum\limits_{k\geq K}\dfrac{1}{k^b}\to 0$ as
$K\to \infty$, we must have that the upper density of $A_K$ goes to
$0$ as $K\to \infty$.\\ This proves that $S_k$ satisfies the density
condition.

  To satisfy condition (b), first observe that every element $s_j$ of
$S_1$ can be thought of as a product of the form $p_1^\frac{c_1
a}{b}p_2^\frac{c_2a}{b}\dots p_k^\frac{c_ka}{b}$ where $c_i$ depend on
$p_i$ and $s_j$ for $i\in[k]$. We will apply Lemma \ref{lem:2} on the set
$S_1$ with $m=b.$\\ To check (i), note that since $S_1$ consist of
\emph{$b$-free} elements, we have $0\leq c_i<b$. Assume $c_i>0$. By
assumption we have gcd$(a,b)=1.$ This implies that $ac_i\neq 0$ (mod
$b$) for $i\in [k].$ \\ For (ii), let if possible,
$s_1=p_1^\frac{c_1a}{b}p_2^\frac{c_2a}{b}\dots p_k^\frac{c_ka}{b}$ and
$s_2=p_1^\frac{d_1a}{b}p_2^\frac{d_2a}{b}\dots p_k^\frac{d_ka}{b}$ be
two distinct elements of $S_1$ satisfying $ac_i=ad_i$ (mod $b$) for
all $i\in [k].$ This implies $a(c_i-d_i)=0$ (mod $b$) for all
$i\in[k]\implies (c_i-d_i)=0$ (mod $b$) for all $i\in[k].$ But this is
not possible, since elements of $S_1$ are \emph{$b$-free}, we must have
$c_i,d_i<b$ for all $i\in[k].$\\ Hence we conclude that $S_1$ is a
\emph{good set}. Since $S_k$ is an integer multiple of $S_1$, so $S_k$
is also linearly independent over $\setQ$.  Thus condition (b) is also
satisfied. This finishes the proof.
\end{proof}
 
Now we will prove Theorem \ref{thm:4} with the help of Proposition \ref{prop:1}. As before, it will be sufficient to show that the sequence $S=(s_n)$ obtained
by rearranging the elements of the set
$\{n_1^{\alpha_1}n_2^{\alpha_2}\dots n_l^{\alpha_l} :n_i\in\setN
\text{ for all } i\in [l]\}$ in an increasing order is {strong sweeping out}.\\
 
\begin{proof}[Proof of Theorem \ref{thm:4}]

  Define
  $$S_1:=\{n_1^{\alpha_1}n_2^{\alpha_2}\dots n_l^{\alpha_l} :n_i\in\setN
  \text{ such that } n_i \text{ is $b_i$-free} \text{ for all }
  i\in[l]\}.$$

  First, let us see that the elements of $S_1$ can be written
  uniquely.  Suppose not. Then there are two different representation
  of an element.
  \begin{equation}
    \label{eq:51} n_1^{\alpha_1}n_2^{\alpha_2}\dots
    n_l^{\alpha_l}=m_1^{\alpha_1}m_2^{\alpha_2}\dots m_l^{\alpha_l}
  \end{equation}

  Let us introduce the following notation:

  \begin{equation}
    \label{eq:59} \overline{b_i}=b_1b_2\dots b_{i-1}b_{i+1}b_l.
  \end{equation}
 
   \eqref{eq:51} implies that
  \begin{equation}
    \label{eq:52} n_1^{a_1\overline{b_1}}n_2^{a_2\overline{b_2}}\dots
    n_l^{a_l\overline{b_l}}=m_1^{a_1\overline{b_1}}m_2^{b_2\overline{b_2}}\dots
    m_l^{a_l\overline{b_l}}
  \end{equation}

  After some cancellation if needed, we can assume that
  $\gcd(n_i,m_i)=1.$ There exists $i$ for which $n_i>1$, otherwise
  nothing to prove. Without loss of generality, let us assume that
  $n_1>1$ and $p$ is a prime factor of $n_1$. Let $p^{c_i}$ be the
  highest power of $p$ that divides $m_i$ or $n_i$ for $i\in[l].$ Let
  us focus on the power of $p$ in equation \eqref{eq:52}.  We have \begin{align*}
      c_1 a_1 \overline{b_1}&=\pm c_2 a_2 \overline{b_2}\pm c_3 a_3
      \overline{b_3}\pm \dots \pm c_l a_l
      \overline{b_l}. \intertext{This implies that}
 c_1a_1\overline{b_1}&=0\ \ (\text{mod}\ b_1).    \end{align*} Since, $(a_1\overline{b_1},b_1)=1$, and $c_1<b_1$,
  we have $c_1=0.$ This is a contradiction!  This proves that
  $n_i=1=m_i$ for all $i\in[l].$ This finishes the proof that every
  element of $S_1$ has a unique representation.  Now, let $(e_k)$ be
  the sequence obtained by arranging the elements of the set
  $\Big\{\prod_{i\leq l} n_i^{b_i}, n_i\in \setN \Big\}$ in an
  increasing order.  Define
  $S_k:=e_kS_1=\{e_kn_1^{\alpha_1}n_2^{\alpha_2}\dots n_l^{\alpha_l}
  :n_i\in\setN \text{ and } n_i \text{ is $b_i$-free} \text{ for all }
  i\in[l]\}.$  Next we will prove that $S =
  \cdot\hspace{-9pt}\bigcup\limits_{k\in \setN} \displaystyle S_k$. 
  One can easily check that $S=\bigcup\limits_{k\in \setN} S_k.$  To
  check the disjointness, we will use a similar argument as above.  Suppose
  on the contrary, there exists $k_1\neq k_2$, such that
  $S_{k_1}=S_{k_2}.$  So we have 

  \begin{equation}
    \label{eq:56} e_{k_1} n_1^{\alpha_1}n_2^{\alpha_2}\dots
    n_l^{\alpha_l}=e_{k_2} m_1^{\alpha_1}m_2^{\alpha_2}\dots
    m_l^{\alpha_l}
  \end{equation} This implies, after using the same notation as
  before  
  \begin{equation}
    \label{eq:50} e_{k_1}^{b_1b_2\dots
      b_l}n_1^{a_1\overline{b_1}}n_2^{a_2\overline{b_2}}\dots
    n_l^{a_l\overline{b_l}}=e_{k_2}^{b_1b_2\dots
      b_l}m_1^{a_1\overline{b_1}}m_2^{b_2\overline{b_2}}\dots
    m_l^{a_l\overline{b_l}}
  \end{equation} Assume without loss of generality that $e_{k_1}>1$,
  $\gcd(e_{k_1},e_{k_2})=1$ and $\gcd(n_i,n_j)=1$ for $i\neq j.$  Let $p$
  be a prime which divides $e_{k_1}$.  Assume $p^{c_i}$ be the highest
  power of $p$ that divides $m_i$ or $n_i$ for $i=1,2,\dots l$ and $p^r$
  be the highest power of $p$ that divides $e_{k_1}$.  Now, we will
  focus on the power of $p$ in  \eqref{eq:50}.  We have,
    \begin{align*} rb_1b_2\dots b_l&=\pm c_1 a_1 \overline{b_1}\pm
      c_2 a_2 \overline{b_2}\pm c_3 a_3 \overline{b_3}\pm \dots \pm c_l a_l
      \overline{b_l}.
\intertext{This implies that for all $i\in[l]$ we have,} c_ia_i\overline{b_i}&=0\ (\text{mod } b_i).
    \end{align*} Since, $(a_i\overline{b_i},b_i)=1$, and $c_i<b_i$,
  we get $c_i=0\text{ for all }i\in[l].$  This is a contradiction to
  the assumption that $p | e_{k_1}.$ Hence we conclude that $e_{k_1}=1.$
  Similarly, $e_{k_2}=1$. This finishes the proof of disjointness of
  $S_k$'s.\\  Now we will check the density condition.  For any sequence
  $A$, we will use the notation $A(N)$ to denote the set $\{a\in A:a\leq
  N\}$.  As before, we will show that the lower density of $B_K :=
  \bigcup\limits _{k\in[K]}S_k$ goes to 1 as $K\to \infty$. Let us look at the complement $C_K:=S
  \setminus B_K$. We want to show that the upper density of $C_K$ goes
  to $0$ as $K \to \infty $.  Since the sets $C_K$ are decreasing, it
  will be sufficient to show that the upper density of the subsequence
  $C_M$ converges to $0$ as $M\to \infty$ where $M={e_K}^{a_1+a_2+\dots
    +a_l}$. But this is equivalent to saying that upper density of $C_M$ converges to $0$ as
  $K\to \infty.$ This will follow from the following two lemmas.
  \begin{lem}[label={lem:3}]{}{} 
        $C_M(N)\subset\bigcup\limits_{i\in[l]}\bigcup\limits_{L\geq
          e_K}D_i^L(N)$  where \\
        $D_i^L(N):=\big\{n_1^{\alpha_1}n_2^{\alpha_2}\dots
        n_{i-1}^{\alpha_{i-1}} (L^{b_i}n_i)^{\alpha_i}
        n_{i+1}^{\alpha_{i+1}} \dots n_l^{\alpha_l}\leq N: n_j \in
        \setN \text{ for }j\in[l] \big\}.$
  \end{lem}

  \begin{proof}
    Let $s=n_1^{\alpha_1}n_2^{\alpha_2}\dots n_l^{\alpha_l}$ belongs
    to $C_{M}(N)$.  Write $n_i=j_i^{b_i}\tilde{n_i}$ where
    $\tilde{n_i}$ is \emph{$b_i$-free}
    for $i\in[l].$ Then we must have $j_i\geq e_K$ for some $i\in [l].$\\
    On the contrary, suppose we have, $j_i<e_K$ for all $i$. Then
    $j_1^{b_1}j_2^{b_2}\dots j_l^{b_l}<{e_K}^{b_1+b_2+\dots+b_l}=M$.
    But this is a contradiction to our hypothesis that $s\in C_M(N).$
  \end{proof}

  \begin{lem}[label={lem:4}]{}{}
    For every fixed $L$ and $i$
    \begin{equation}
      \label{eq:63} \dfrac{\#D_i^L(N)}{\#S(N)}\leq \frac{1}{L^{b_i}}.
    \end{equation}
  \end{lem}
  \begin{proof}
    Fix $n_j=r_j$ for $j=1,2,\dots ,i-1,i+1,\dots,l.$ Then
    \begin{align*} &\#\{r_1^{\alpha_1}r_2^{\alpha_2}\dots
      r_{i-1}^{\alpha_{i-1}}
      (L^{b_i}n_i)^{\alpha_i}r_{i+1}^{\alpha_{i+1}}
                     \dots r_l^{\alpha_l}\leq N: n_i \in \setN\}\\
                   &=\Big\lfloor{\frac{N^\frac{1}{\alpha_i}}{\kappa
                     L^{b_i}}}\Big\rfloor \text{ where }
                     \kappa=\big(r_1^{\alpha_1}r_2^{\alpha_2}\dots
                     r_{i-1}^{\alpha_{i-1}}r_{i+1}^{\alpha_{i+1}}
                     \dots
                     r_l^{\alpha_l}\big)^\frac{1}{\alpha_i}\text{ and }\\
                   &\#\big\{r_1^{\alpha_1}r_2^{\alpha_2}\dots
                     r_{i-1}^{\alpha_{i-1}} (n_i)^{\alpha_i}
                     r_{i+1}^{\alpha_{i+1}}\dots r_l^{\alpha_l}\leq N:
                     n_i
                     \in \setN\big\}\\
                   &=\Big\lfloor{\frac{N^\frac{1}{\alpha_i}}{\kappa}}\Big\rfloor
    \end{align*}

    The proof follows by observing the following two
    equalities:
    \begin{align*}
        \#S{(N)}&=\#\big\{n_1^{\alpha_1}n_2^{\alpha_2}\dots
        n_l^{\alpha_l}\leq N: n_j \in \setN \text{ for }j\in[l] \big\}\\
        &= \sum\#\{r_1^{\alpha_1}r_2^{\alpha_2}\dots
        r_{i-1}^{\alpha_{i-1}} (n_i)^{\alpha_i}
        r_{i+1}^{\alpha_{i+1}}\dots r_l^{\alpha_l}\leq N: n_i
        \in \setN\}
        \intertext{and}
        \#D_i^L(N)&=\#\{n_1^{\alpha_1}n_2^{\alpha_2}\dots
        n_{i-1}^{\alpha_{i-1}} (L^{b_i}n_i)^{\alpha_i}
        n_{i+1}^{\alpha_{i+1}} \dots n_l^{\alpha_l}\leq N: n_j \in
        \setN \text{ for }j\in[l] \}\\ &=
        \sum\#\{r_1^{\alpha_1}r_2^{\alpha_2}\dots
        r_{i-1}^{\alpha_{i-1}} (L^{b_i}n_i)^{\alpha_i}
        r_{i+1}^{\alpha_{i+1}}\dots r_l^{\alpha_l}\leq N: n_i \in
        \setN\}
    \end{align*}
    where the summations are taken over all the $(l-1)$tuple
    $(r_1,r_2,\dots,r_{i-1},r_{i+1},\dots,r_l)$ such that their
    products
    $r_1^{\alpha_1}r_2^{\alpha_2}\dots r_{i-1}^{\alpha_{i-1}}
    r_{i+1}^{\alpha_{i+1}}$ are distinct.
  \end{proof}
  Hence the relative upper density of the set $C_{M}$ in $S$ is
  $\leq \displaystyle \sum_{i\in [l]}\sum_{L\geq e_K}\frac{1}{L^{b_i}}$
  which goes to $0$ as $K\to \infty.$

  Now we will show that each $S_k$ is linearly independent. It will be sufficient to show that $S_1$ is linearly
  independent.  After realizing every element $s_j$ of $S_1$ as
  $\prod\limits_{p\in \mathcal{P}} p^\frac{v_{j,p}}{b_1b_2\dots b_l}$,
  we will apply Lemma \ref{lem:2} on
  $S_1$ with $m=b_1b_2\dots b_l.$
  \begin{enumerate}[(i)]
  \item First let us show that for any given $j_0\in J$, and prime
    $p_0\in \mathcal{P}$, if $v_{j_0,p_0}\neq 0$, then
    $v_{j_0,p_0}\not\equiv 0$ (mod $m$). By definition of the set
    $S_1$, $v_{j_0,p_0}$ can be written as
    $v_{j_0,p_0}=c_1a_1\overline{b_2}+c_2a_2\overline{b_2}+\dots+c_la_l\overline{b_l}$
    ,where $0\leq c_i<b_i$, and $c_i$ depends on $s_{j_0,p_0}$ for
    $i\in[l]$. Here $\overline{b_i}$ is same as in  \eqref{eq:59}.\\ On
    the
    contrary, suppose $v_{j_0,p_0}\equiv 0$ (mod $m$). Hence we have\\
      \begin{align*}
        c_1a_1\overline{b_1}+c_2a_2\overline{b_2}+\dots+c_la_l\overline{b_l}&\equiv
        0 \text{ (mod }m).\intertext{This means}
        c_1a_1\overline{b_1}+c_2a_2\overline{b_2}+\dots+c_la_l\overline{b_l} &\equiv
        0 \text{ (mod }b_i) \text{ for all }i\in [l]. \intertext{But this implies}
        c_ia_i\overline{b_i}&\equiv 0 \text{ (mod }b_i) \text{ for all }i\in
        [l],\intertext{ implying that } c_i&=0 \text{ for all }i \in [l].\intertext{ So, we get }
        v_{j_0,p_0}&=0.
      \end{align*} This is a contradiction! Hence we conclude that if
    $v_{j_0,p_0}\neq 0$, then $v_{j_0,p_0}\not\equiv 0$ (mod $m$).
  \item Now, let us assume that $v_{j,p}\equiv v_{k,p}$ (mod $m)$, for
    all $p\in \mathcal{P}$.\\ Write
    $v_{j,p}=c_1a_1\overline{b_2}+c_2a_2\overline{b_2}+\dots+c_la_l\overline{b_l}$
    and
    $v_{k,p}=d_1a_1\overline{b_2}+d_2a_2\overline{b_2}+\dots+d_la_l\overline{b_l}$
    where $0\leq c_i,d_i<b$, and $c_i$ depends on $p,s_j$ and $d_i$
    depends on $p,s_k$ for all $i\in[l]$. By assumption, we have
\begin{align*}
c_1a_1\overline{b_2}+c_2a_2\overline{b_2}+\dots+c_la_l\overline{b_l}&\equiv d_1a_1\overline{b_2}+d_2a_2\overline{b_2}+\dots+d_la_l\overline{b_l}
        \text{ (mod }m)\intertext{ This implies that } (c_i-d_i)a_i\overline{b_i}&\equiv
        0 \text{ (mod }b_i) \text{ for all }i\in [l].\intertext{ Hence we get}
        (c_i-d_i)&=0 \text{ for all }i\in [l].\intertext{So, we conclude that}
        v_{j,p}=v_{k,p} \text{ for all } p\in \mathcal{P}.
      \end{align*} But we know that every element of $S_1$ has a
    unique representation. Hence, we get $j=k$.\\ Thus we conclude that $S$
    is a \emph{good set} and hence linearly independent over the rational.
  \end{enumerate} This completes the proof.
\end{proof}

 \section{The case when $\alpha$ is irrational: two open problems}
 \vspace{.3cm} 
 We have proved that the sequence $(n^\alpha)$ is
{strong sweeping out} when $\alpha$ is a \emph{non-integer rational
number}. For the case of irrational $\alpha$, it is known from \autocite{BBB} and \autocite{JW} that
$(n^\alpha)$ is {strong sweeping out} for all but countably many $\alpha$. However, the following questions are still open along this
line:
\begin{problem}[label={prob:1}]{}{}
 Find an explicit irrational $\alpha$, for which $(n^\alpha)$ is
strong sweeping out.
\end{problem}
\begin{problem}[label={prob:2}]{}{}
 Is it true that for all irrational $\alpha$, $(n^\alpha)$ is {pointwise bad} for $L^2$? If so, then is it true that $(n^\alpha)$ is {strong sweeping out}?
\end{problem}
 
 If one can find an explicit $\alpha$, for which $(n^\alpha)$ is linearly independent over the rational, then obviously it will answer Problem \ref{prob:1}. In fact, if we can find a suitable subsequence which is linearly independent over $\setQ$, then we can apply Proposition \ref{prop:1} to answer Problem \ref{prob:1}.\\ It might be possible to apply Theorem \ref{thm:4} and some density argument to handle Problem \ref{prob:2}.
 
 \section{Acknowledgements.}
The result is part of the author's thesis. He
is indebted to his advisor, Professor Máté Wierdl, for suggesting the problem and giving valuable inputs. He is supported by the National Science Foundation under grant
number DMS-1855745.

\printbibliography

\end{document}